\input ./style/arxiv-vmsta.cfg
\documentclass[numbers,compress,v1.0.1]{vmsta}

\volume{2}
\issue{2}
\pubyear{2015}
\firstpage{131}
\lastpage{146}
\doi{10.15559/15-VMSTA27}

\startlocaldefs

\allowdisplaybreaks

\DeclareMathOperator*{\Prob}{{\sf P}}
\DeclareMathOperator*{\ME}{{\sf E}}
\DeclareMathOperator*{\sign}{sign}
\DeclareMathOperator*{\zonefun}{z_1}
\DeclareMathOperator*{\cdf}{cdf}
\newcommand{\econst}{\mathrm{e}}
\newcommand{\der}{\mathrm{d}}
\newcommand{\dert}{\der t}
\newcommand{\Xobs}{X^{\rm obs}}
\newcommand{\varmode}{M}

\newtheorem{thm}{Theorem}
\newtheorem{lemma}[thm]{Lemma}
\newtheorem{corollary}[thm]{Corollary}
\endlocaldefs

\begin{document}
\begin{frontmatter}

\title{Identifiability of logistic regression with homoscedastic error:
Berkson model}

\author{\inits{S.}\fnm{Sergiy}\snm{Shklyar}}\email{shklyar@univ.kiev.ua}
\address{Taras Shevchenko National University of Kyiv, Ukraine}

\markboth{S. Shklyar}{Identifiability of logistic regression with
homoscedastic error: Berkson model}

\begin{abstract}
We consider the Berkson model of logistic regression
with Gaussian and homosce\-dastic error in regressor.
The measurement error variance can be either known or unknown.
We deal with both functional and structural cases.
Sufficient conditions for
identifiability of regression coefficients are presented.

Conditions for identifiability of the model are studied.
In the case where the error variance is known,
the regression parameters are identifiable
if the distribution of the observed regressor is not concentrated
at a single point.
In the case where the error variance is not known,
the regression parameters are identifiable
if the distribution of the observed regressor
is not concentrated at three (or less) points.

The key analytic tools are relations between
the smoothed logistic distribution function
and its derivatives.
\end{abstract}

\begin{keyword}
Logistic regression \sep
binary regression \sep
errors in variables \sep
Berkson model \sep
regression calibration model

\MSC[2010] 62J12
\end{keyword}

\received{20 May 2015}
\revised{19 June 2015}
\accepted{20 June 2015}
\publishedonline{7 July 2015}
\end{frontmatter}

\section{Introduction}\label{sec:intro}
\paragraph{Statistical model}
Consider logistic regression with Berkson-type error in the
explanatory variable. One trial is distributed as follows.
$\Xobs_n$ is the observed (or assigned) surrogate regressor.
The true regressor is $X_n = \Xobs_n + U_n$,
where the error $U_n\sim N(0,\tau^2)$ is
independent of $\Xobs_n$.
The response $Y_n$ is a binary random variable and attains
either $0$ or $1$ with
\[
\Prob\bigl(Y_n {=} 1 \bigm|\Xobs_n, X_n\bigr) =
\frac{\exp(\beta_0 + \beta_1 X_n)}{1 + \exp(\beta_0 + \beta_1 X_n)} .
\]

We consider both functional model and structural model.
In the functional one, $\Xobs_n$ are nonrandom variables, and
in the structural one, $\Xobs_n$ are i.i.d.,
and therefore in the latter model, $(\Xobs_n, X_n, Y_n)$ are i.i.d.\@
random triples.

The couples $(\Xobs_n, Y_n)$, $n=1,\ldots,N$, are observed.
Vector $\vec\beta= (\beta_0, \beta_1)^\top$ is a parameter of interest.

The error variance $\tau^2$ can be either known or unknown, and
we consider both cases.
The conditions for identifiability of the model (or of the
parameter~$\vec\beta$)
are \xch{presented.}{presented}

\paragraph{Overview}
Berkson models of logistic regression and probit regression were set up
in \citet{Burr83}. For probit regression, it is shown that the
introduction of Berkson-type error is equivalent to augmentation of
regression parameters. As a consequence, the Berkson model of probit
regression is identifiable if $\tau^2$ is known and is not identifiable
if $\tau^2$ is not known.

The identifiability of the classical model was studied by \citet
{Kuech1995}. He assumes that both the regressor and measurement error
are normally distributed. Then univariate logistic regression is
identifiable (here $\tau^2$ can be unknown), and multiple logistic
regression is not identifiable. Our results can be proved similarly to
\cite{Kuech1995} if we assume that the distribution of the surrogate
regressor $\Xobs$ has an unbounded support.\looseness=-1

For classification of errors-in-variables regression models and various
estimation methods, see the monograph by \citet{CaRuSt}.

Identifiability of the statistical model can be used in the proof of
consistency of the estimator. For known $\tau^2$, the strong
consistency of the maximum likelihood estimator is obtained by \citet
{ShklyarS}. But if $\tau^2$ is not known, the maximum likelihood
estimator seems to be unstable
(see discussion in \cite{CaRuSt} or~\cite{Kuech1995}).

\section{Convolution of logistic function with normal density}
Consider the function
%
\begin{equation}
\label{eq:def-L0} L_0 \bigl(x, \sigma^2 \bigr) = \ME
\frac{\exp(x - \xi)}{1 + \exp(x - \xi)}, \quad \xi\sim N \bigl(0,\sigma^2 \bigr), \ x \in
\mathbb{R}, \ \sigma^2 \ge0 ,
\end{equation}
that is, $L_0(x, 0) = \econst^x / (1 + \econst^x)$ and
\[
L_0 \bigl(x, \sigma^2 \bigr) = \frac{1}{\sqrt{2 \pi} \sigma} \int
_{-\infty}^{\infty} \frac{\exp(x - t)}{1 + \exp(x - t)} \econst^{- t^2 / (2 \sigma^2)}
\, \dert\quad\mbox{for} \ \sigma^2 > 0 .
\]

Denote the derivatives w.r.t.\@ $x$
%
\begin{equation}
\label{eq:def-Lk} L_k \bigl(x, \sigma^2 \bigr) =
\frac{\partial^k}{\partial x^k} L_0 \bigl(x, \sigma^2 \bigr) .
\end{equation}

Differentiation of $L_k(x, \sigma^2)$ with respect to the second
argument is described in Appendix \ref{Appendix-A}.

The distribution of $Y_i$ given $\Xobs_i$ is
\begin{align}
\Prob\bigl[ Y_i = 1 \bigm|\Xobs_i\bigr] &= \ME \bigl[
\Prob\bigl[Y_i = 1 \bigm|\Xobs_i, X_i\bigr] \bigm|
\Xobs_i \bigr]
\nonumber
\\
&= \ME \biggl[ \frac{\exp(\beta_0 + \beta_1 X_i)}{
1+\exp(\beta_0 + \beta_1 X_i)} \biggm| \Xobs_i \biggr] = L \bigl(
\beta_0 + \beta_1 \Xobs_i, \:
\beta_1^2 \tau^2 \bigr) \label{eq:condprobYiXobsi}
\end{align}
since $[\beta_0 + \beta_1 X_i \mid\Xobs_i] \sim N(\beta_0 + \beta_1
\Xobs_i, \: \beta_1^2 \tau^2)$.

\section[Identifiability when error variance known]
{Identifiability when $\tau^2$ is known}
\begingroup\abovedisplayskip=11pt\belowdisplayskip=11pt
\begin{thm}\label{thm-func-known-i}
If in the functional model not all $\Xobs$ are equal, then the model is
identifiable.
\end{thm}

\begin{proof}
Suppose that for two values of parameters
$\vec\beta^{(1)} = (\beta^{(1)}_0, \beta^{(1)}_1)$ and
$\vec\beta^{(2)} = (\beta^{(2)}_0, \beta^{(2)}_1)$,
$\vec\beta^{(1)} \neq\vec\beta^{(2)}$,
the distributions of observations are equal.
Then for all $i = 1, 2,\ldots, N$,
\begin{gather*}
\Prob\nolimits_{\vec\beta^{(1)}} (Y_i = 1) = \Prob
\nolimits_{\vec\beta^{(2)}} (Y_i = 1),
\\
L_0 \bigl(\beta^{(1)}_0 +
\beta^{(1)}_1 \Xobs_i, \bigl(
\beta_1^{(1)}\bigr)^2 \tau^2 \bigr)
= L_0 \bigl(\beta^{(2)}_0 +
\beta^{(2)}_1 \Xobs_i, \bigl(
\beta_1^{(2)}\bigr)^2 \tau^2 \bigr)
.
\end{gather*}
However, by Lemma 4.1 from \cite{ShklyarS}
the equation
\[
L_0 \bigl(\beta^{(1)}_0 +
\beta^{(1)}_1 x, \bigl(\beta_1^{(1)}
\bigr)^2 \tau^2 \bigr) = L_0 \bigl(
\beta^{(2)}_0 + \beta^{(2)}_1 x,
\bigl(\beta_1^{(2)}\bigr)^2 \tau^2
\bigr)
\]
has no more than one solution $x$.
Hence, all $\Xobs_i$ are equal.
\end{proof}

By definition the degenerate distribution is the distribution
concentrated at a single point. For the next theorem, see the proof of
Theorem 5.1 in \cite{ShklyarS}.
\begin{thm}[\cite{ShklyarS}]
If in the structural model the distribution of $\Xobs_1$ is not degenerate,
then the parameter $\vec\beta$ is identifiable.\vspace*{-2pt}
\end{thm}

\section[Identifiability when error variance unknown]
{Identifiability when $\tau^2$ is unknown}\vspace*{-2pt}
For fixed $\sigma^2$, the function $L_0(x, \sigma^2)$ is a bijection
$\mathbb{R} \to(0,\: 1)$.
{Hence}, for fixed $\sigma_1^2$ and $\sigma_2^2$,
the relation
%
\begin{equation}
L_0 \bigl(y, \sigma_1^2 \bigr) =
L_0 \bigl(x, \sigma_2^2 \bigr) \label{eq-L1700}
\end{equation}
\endgroup
sets the bijection $\mathbb{R} \to\mathbb{R}$; see Fig. \ref{fig:4xy}.

\begin{lemma}
\label{lem-d2-eq1700}
For fixed $\sigma_1^2 \ge0$ and $\sigma_2^2 \ge0$,
the sign of the second derivative of the implicit function \eqref{eq-L1700}
is
\[
\sign \biggl( \frac{\der^2 y}{\der x^2} \biggr) = \sign \bigl(\sigma_2^2
- \sigma_1^2 \bigr) \sign(x).
\]
\end{lemma}

\begin{proof}
Differentiating (\ref{eq-L1700}), we get
\begin{align*}
L_1 \bigl(y, \sigma_1^2 \bigr) \, \der y&
= L_1 \bigl(x, \sigma_2^2 \bigr) \, \der x;
\\
\frac{\der y}{\der x} &= \frac{L_1(x, \sigma_2^2)} {L_1(y, \sigma_1^2)}.
\end{align*}
Then
\begin{align*}
\frac{\der^2 y}{\der x^2} & = \frac{L_2(x, \sigma_2^2) L_1(y, \sigma_1^2) -
L_1(x, \sigma_2^2) L_2(y, \sigma_1^2) \frac{\der y}{\der x}}{
L_1(y,\sigma_1^2)^2}
\\
&= \frac{L_2(x, \sigma_2^2) L_1(y,\sigma_1^2)^2 -
L_1(x, \sigma_2^2)^2 L_2(y, \sigma_1^2)}{
L_1(y,\sigma_1^2)^3}
\\
&= \biggl( \frac{L_2(x,\sigma_2^2)}{L_1(x,\sigma_2^2)^2} - \frac{L_2(y,\sigma_1^2)}{L_1(y,\sigma_1^2)^2} \biggr) \cdot
\frac{L_1(x,\sigma_2^2)^2}{L_1(y,\sigma_1^2)}.
\end{align*}
Thus,
%
\begin{equation}
\sign \biggl(\frac{\der^2 y}{\der x^2} \biggr) = \sign \biggl( \frac{L_2(x,\sigma_2^2)}{L_1(x,\sigma_2^2)^2} -
\frac{L_2(y,\sigma_1^2)}{L_1(y,\sigma_1^2)^2} \biggr) . \label{eq-1750}
\end{equation}

\begin{figure}
\includegraphics{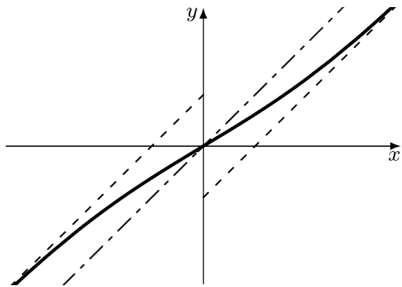}
\caption{The plot to equation $L_0(y, \sigma_1^2) = L_0(x, \sigma_2^2)$
for $\sigma_1^2 < \sigma_2^2$}\label{fig:4xy}
\end{figure}

Denote by $\mu(z, \sigma^2)$ the solution to the equation
$L_0(\mu,\sigma^2) = z$. Note that as $L_0(x,\sigma^2)$ is
the cdf of a symmetric distribution,
$\sign(L_0(x,\sigma^2)-0.5) = \sign(x)$. Therefore,
$\sign(\mu(z,\sigma^2)) = \sign(z-0.5)$.
Find the derivative
\[
\frac{\der}{\der v} \biggl( \frac{L_2(\mu(z,v),v)}{L_1(\mu(z,v),v)^2} \biggr)
\]
for fixed $z$.
By the implicit function theorem,
\[
\frac{\der\mu(z,v)}{\der v} = - \frac{L_2(\mu(z,v),v)}{2 L_1(\mu(z,v),v)};
\]
also,
\begin{align*}
\frac{\partial}{\partial x} \biggl(\frac{L_2(x,v)}{L_1(x,v)^2} \biggr)& = \frac{L_3(x,v) L_1(x,v) - 2 L_2(x,v)^2}{L_1(x,v)^3},
\\
\frac{\partial}{\partial v} \biggl(\frac{L_2(x,v)}{L_1(x,v)^2} \biggr) &= \frac{L_4(x,v) L_1(x,v) - 2 L_2(x,v) L_3(x,v)}{2 L_1(x,v)^3}.
\end{align*}
Then
\begin{align*}
\frac{\der}{\der v} \biggl( \frac{L_2(\mu(z,v),v)}{L_1(\mu(z,v),v)^2} \biggr) &= - \frac{L_2}{2L_1}
\cdot \frac{L_3 L_1 - 2 L_2^2}{L_1^3} + \frac{L_4 L_1 - 2 L_2 L_3}{2 L_1^3}
\\
&= \frac{L_4 L_1^2 - 3 L_3 L_2 L_1 + 2 L_2^3}{2 L_1^4} ,
\end{align*}
where $L_k$ are evaluated at the point $(\mu(z,v),v)$.
By Lemma~\ref{lem-key-ineq4},
\[
\sign \biggl( \frac{\der}{\der v} \biggl( \frac{L_2(\mu(z,v),v)}{L_1(\mu(z,v),v)^2} \biggr) \biggr) =
\sign \bigl(\mu(z,v) \bigr) = \sign(z-0.5) .
\]

The function $v \mapsto\frac{L_2(\mu(z,v),v)}{L_1(\mu(z,v),v)^2}$
is monotone (it is increasing for $z>0.5$ and decreasing for $z<0.5$).
For $x$ and $y$ satisfying (\ref{eq-L1700}),
\[
x = \mu \bigl(z,\sigma_2^2 \bigr) \quad\mbox{and}\quad y
= \mu \bigl(z,\sigma_1^2 \bigr)
\]
with $z=L_0(y,\sigma_1^2)=L_0(x,\sigma_2^2)$;
note that $\sign(z-0.5)=\sign(x)$.
Then
\[
\sign \biggl( \frac{L_2(x,\sigma_2^2)}{L_1(x,\sigma_2^2)^2} - \frac{L_2(y,\sigma_1^2)}{L_1(y,\sigma_1^2)^2} \biggr) = \sign \bigl(
\sigma_2^2 - \sigma_1^2 \bigr)
\sign(x),
\]
and with (\ref{eq-1750}),
we can obtain the desired equality
\begin{equation*}
\sign \biggl(\frac{\der^2 y}{\der x^2} \biggr) = \sign \bigl(\sigma_2^2
- \sigma_1^2 \bigr) \sign(x). \qedhere
\end{equation*}
\end{proof}

\begin{lemma}
\label{lem-eq-L0-sigs}
The equation
%
\begin{equation}
L_0 \bigl(\beta^{(1)}_0 +
\beta^{(1)}_1 x, \: \sigma_1^2
\bigr) = L_0 \bigl(\beta^{(2)}_0 +
\beta^{(2)}_1 x, \: \sigma_2^2
\bigr) \label{eq-L1703}
\end{equation}
has no more than three solutions, unless either
%
\begin{equation}
\label{eq:excase1} \vec\beta^{(1)} = \vec\beta^{(2)} \quad
\mbox{and}\quad \sigma_1^2 = \sigma_2^2
\end{equation}
or
%
\begin{equation}
\label{eq:excase2} \beta^{(1)}_1 = \beta^{(2)}_1
= 0 \quad\mbox{and}\quad L_0 \bigl(\beta_0^{(1)},
\sigma_1^2 \bigr) = L_0 \bigl(
\beta_0^{(2)}, \sigma_2^2 \bigr) .
\end{equation}
In exceptional cases \eqref{eq:excase1} and \eqref{eq:excase2},
equation~\eqref{eq-L1703} is an identity.
\end{lemma}


\begin{proof}
The proof has the following idea: if a twice differentiable function
$y(x)$ satisfies \eqref{eq-L1700}, then the plot of the function
either is a straight line (if $\sigma_1^2 = \sigma_2^2$)
or intersects any straight line at no more than three points.

Consider four cases.

\textit{Case 1.} $\sigma_1^2 = \sigma_2^2$.
Since the function $L_0(z, \sigma^2)$ is strictly increasing in $z$,
Eq.~\eqref{eq-L1703} is equivalent to
\[
\beta_0^{(1)} + \beta_1^{(1)} x =
\beta_0^{(2)} + \beta_1^{(2)} x .
\]
Equation \eqref{eq-L1703} has only one solution if
$\beta_1^{(1)} \neq\beta_1^{(2)}$;
it is an identity if $\vec\beta^{(1)} = \vec\beta^{(2)}$,
and it has no solutions if $\beta_1^{(1)} = \beta_1^{(2)}$
but $\beta_0^{(1)} \neq\beta_0^{(2)}$.

\textit{Case 2.} $\beta_1^{(2)} = 0$ and $\beta_1^{(1)} \neq0$.
For any fixed $\sigma^2$, the function $z \mapsto L_0(z, \sigma^2)$
is a bijection $\mathbb{R} \to(0,\, 1)$.
Denote the inverse function $\mu(Z, \sigma^2)$:
$L_0( z, \sigma^2) = Z$ if and only if $z = \mu(Z, \sigma^2)$.
Equation \eqref{eq-L1703} has a unique solution
\[
x = \frac{\mu(L_0(\beta_0^{(2)}, \sigma_2^2), \sigma_1^2) - \beta_0^{(1)}}{
\beta_1^{(1)}}.
\]

\textit{Case 3.} $\beta_1^{(2)} = \beta_1^{(1)} = 0$.
Neither side of \eqref{eq-L1703} depends on $x$.
Equation \eqref{eq-L1703} becomes
$L_0(\beta_0^{(1)}, \sigma_1^2)=
L_0(\beta_0^{(2)}, \sigma_2^2)$.
Equation \eqref{eq-L1703} either holds for all $x$
or does not hold for any $x$.

\textit{Case 4.} $\sigma_1^2 \neq\sigma_2^2$ and $\beta_1^{(2)} \neq0$.
Make a linear variable substitution: denote
$z_2 = \beta_0^{(2)} + \beta_1^{(2)} x$.
Then Eq.~\eqref{eq-L1703} becomes
%
\begin{equation}
\label{eq-L1703sbst} L_0 \biggl(\beta_0^{(1)} +
\frac{\beta_1^{(1)}}{\beta_1^{(2)}} \cdot \bigl(z_2 - \beta_0^{(2)}
\bigr), \: \sigma_1^2 \biggr) = L_0
\bigl(z_2, \sigma_2^2 \bigr) .
\end{equation}

Define the function $\zonefun(z_2)$ from the equation
\[
L_0 \bigl(\zonefun(z_2), \sigma_1^2
\bigr) = L_0 \bigl(z_2, \sigma_2^2
\bigr) .
\]
The function $\zonefun(z_2) : \mathbb{R} \to\mathbb{R}$
is implicitly defined by Eq.~\eqref{eq-L1700}: there the equality holds
if and only if
$y = \zonefun(x)$.
Hence, the function $\zonefun(z_2)$ satisfies Lemma~\ref{lem-d2-eq1700}.
Equation \eqref{eq-L1703sbst} is equivalent to
%
\begin{equation}
\zonefun(z_2) - \beta_0^{(1)} -
\frac{\beta_1^{(1)}}{\beta_1^{(2)}} \cdot \bigl(z_2 - \beta_0^{(2)}
\bigr) = 0. \label{eq-1703-equiv2}
\end{equation}
By Lemma~\ref{lem-d2-eq1700},
\begin{align*}
&\sign \biggl(\frac{\der^2}{\der z_2^2} \biggl( \zonefun(z_2) -
\beta_0^{(1)} - \frac{\beta_1^{(1)}}{\beta_1^{(2)}} \cdot
\bigl(z_2 - \beta_0^{(2)} \bigr) \biggr) \biggr)
\\
&\quad{}= \sign \biggl(\frac{\der^2 \zonefun(z_2)} {\der z_2^2} \biggr) = \sign \bigl(
\sigma_2^2 - \sigma_1^2 \bigr)
\sign(z_2) .
\end{align*}
Then the derivative of the left-hand size of \eqref{eq-1703-equiv2}
%
\begin{equation}
\label{eq-1703-equiv2-d1} \frac{\der}{\der z_2} \biggl( \zonefun(z_2) -
\beta_0^{(1)} - \frac{\beta_1^{(1)}}{\beta_1^{(2)}} \cdot
\bigl(z_2 - \beta_0^{(2)} \bigr) \biggr)
\end{equation}
is strictly monotone on both intervals $(-\infty,\: 0]$ and $[0,\:
{+}\infty)$,
and hence \eqref{eq-1703-equiv2-d1} attains 0 no more than at two points.
Then the left-hand side of \eqref{eq-1703-equiv2} has no more than
three intervals
of monotonicity, and Eq.~\eqref{eq-1703-equiv2} has no more than
three solutions. Equation \eqref{eq-L1703} has the same number of solutions.
\end{proof}


\begin{thm}
\label{thm-indentif-func}
If in the functional model there are four different $\Xobs$,
then the parameters $\vec\beta$ and $\beta_1^2\tau^2$ are identifiable.
\end{thm}
\begin{proof}
Suppose that there are two sets of parameters $(\vec\beta^{(1)},(\tau^{(1)})^2)$
and $(\vec\beta^{(2)},\allowbreak(\tau^{(2)})^2)$
that for a given sample of the surrogate, the regressors
$\{X_{0n},\allowbreak\; n=1,\ldots,N\}$
provide the same distribution of $Y_n$, $n{=}1,\ldots,N$.
Then for all $n=1,\ldots, N$,
\begin{align*}
\Prob\nolimits_{\vec\beta^{(1)}, (\tau^{(1)})^2} (Y_n = 1) &= \Prob
\nolimits_{\vec\beta^{(2)},(\tau^{(2)})^2} (Y_n = 1);
\\
L_0 \bigl(\beta_0^{(1)} +
\beta_1^{(1)} \Xobs_n, \: \bigl(
\beta_1^{(1)}\bigr)^2 \bigl(\tau^{(1)}
\bigr)^2 \bigr) &= L_0 \bigl(\beta_0^{(2)}
+ \beta_1^{(2)} \Xobs_n, \: \bigl(
\beta_1^{(2)}\bigr)^2 \bigl(\tau^{(2)}
\bigr)^2 \bigr).
\end{align*}
The equation
\[
L_0 \bigl(\beta_0^{(1)} +
\beta_1^{(1)} x, \: \bigl(\beta_1^{(1)}
\bigr)^2 \bigl(\tau^{(1)}\bigr)^2 \bigr) =
L_0 \bigl(\beta_0^{(2)} +
\beta_1^{(2)} x, \: \bigl(\beta_1^{(2)}
\bigr)^2 \bigl(\tau^{(2)}\bigr)^2 \bigr)
\]
has at least four solutions. Then by Lemma~\ref{lem-eq-L0-sigs}
either
\[
\vec\beta^{(1)} = \vec\beta^{(2)} \quad\mbox{and}\quad \bigl(
\beta_1^{(1)}\bigr)^2 \bigl(\tau^{(1)}
\bigr)^2 = \bigl(\beta_1^{(2)}
\bigr)^2 \bigl(\tau^{(2)}\bigr)^2,
\]
or
%
\begin{equation}
\beta_1^{(1)} = \beta_2^{(2)} = 0
\quad\mbox{and}\quad L_0 \bigl(\beta_0^{(1)} ,
\: \bigl(\beta_1^{(1)}\bigr)^2 \bigl(
\tau^{(1)}\bigr)^2 \bigr) = L_0 \bigl(
\beta_0^{(2)} ,\: \bigl(\beta_1^{(2)}
\bigr)^2 \bigl(\tau^{(2)}\bigr)^2 \bigr) .
\label{eq-L1941}
\end{equation}
In the latter alternative,
\begin{gather*}
\bigl(\beta_1^{(1)}\bigr)^2 \bigl(
\tau^{(1)}\bigr)^2 = \bigl(\beta_1^{(2)}
\bigr)^2 \bigl(\tau^{(2)}\bigr)^2 = 0 \quad
\mbox{and}\quad \beta_0^{(1)} = \beta_0^{(2)}
\end{gather*}
since $L_0(b_0, 0) = \frac{1}{1 + \econst^{-b_0}}$ is a strictly increasing
function in $b_0$.
\end{proof}


\begin{thm}
\label{prop-indent-scalunk}
If in the structural model the distribution of $X_0$
is not concentrated at three \emph{(}or less\emph{)} points,
then the parameters $\vec\beta$ and $\beta_1^2\tau^2$ are identifiable.
\end{thm}

\begin{proof}
Suppose that there are two sets of parameters $(\vec\beta^{(1)},(\tau^{(1)})^2)$
and $(\vec\beta^{(2)},\allowbreak(\tau^{(2)})^2)$
for which the same bivariate distribution of $(\Xobs_1, Y_1)$ is obtained.
The random variable $\Prob[Y_1 = 1 \mid\Xobs_1]$ satisfies Eq.~\eqref
{eq:condprobYiXobsi} almost surely for each set of parameters.
Hence, the equality
\[
L_0 \bigl(\beta_0^{(1)} +
\beta_1^{(1)} \Xobs_1, \: \bigl(
\beta_1^{(1)}\bigr)^2 \bigl(\tau^{(1)}
\bigr)^2 \bigr) = L_0 \bigl(\beta_0^{(2)}
+ \beta_1^{(2)} \Xobs_1, \: \bigl(
\beta_1^{(2)}\bigr)^2 \bigl(\tau^{(2)}
\bigr)^2 \bigr)
\]
holds almost surely.
The rest of the proof is the same as in Theorem \ref{thm-indentif-func}.
\end{proof}


\appendix

\section[Differentiation of the function Lk]{Differentiation of $L_k(x,
\sigma^2)$}\label{Appendix-A}
Consider the sum of two independent random variables
$\zeta= \lambda+ \xi$, where $\lambda$ has the logistic distribution
\[
\Prob(\lambda\le x) = \frac{\exp(x)}{1+\exp(x)}, \quad x\in\mathbb{R},
\]
and $\xi\sim N(0,\sigma^2)$.
We allow $\sigma^2=0$, and then $\xi= 0$ almost surely.

The function $L_0(x, \sigma^2)$ defined in \eqref{eq:def-L0} is
the cdf of $\zeta$,
and the function $L_1(x, \sigma^2)$ defined in \eqref{eq:def-Lk} is
the pdf of $\zeta$.

The partial derivatives of $L_k(x, v)$ are
\begingroup\abovedisplayskip=12pt\belowdisplayskip=12pt
\[
\frac{\partial}{\partial x} L_k(x, v) = L_{k+1}(x, v), \qquad
\frac{\partial}{\partial v} L_k(x, v) = \frac{1}2 L_{k+2}(x,
v);
\]
see the proof in \cite[Section 2]{ShklyarS}.
The functions $L_k(x, v)$ are infinitely differentiable
and bounded on
$\mathbb{R} \times[0, +\infty)$.

Since the distribution of $\zeta$ is symmetric,
\[
L_k \bigl(-x, \sigma^2 \bigr) = (-1)^{k-1}
L_k \bigl(x, \sigma^2 \bigr), \quad k\ge1,
\]
that is, $L_1(x, \sigma^2)$ and $L_3(x, \sigma^2)$ are even functions
in $x$,
and $L_2(x, \sigma^2)$ and $L_4(x, \sigma^2)$ are odd functions in $x$.

\section{The key inequality}\label{Appendix-B}
The next lemma is similar to Lemma 2.1 in \cite{ShklyarS}.
Hence, the proof is brief; see \cite{ShklyarS} for details.
\begin{lemma}
Let $\xi$ and $\eta$ be two independent random variables, where
$\xi\sim N(0,1)$.
Denote $\zeta= \xi+ \eta$ and let $p_\zeta(z)$ be the pdf of $\zeta$.
Then
\[
\frac{\der^3}{\der z^3} \bigl(\ln p_\zeta(z) \bigr) = \mu_3[\eta
\mathrel| \zeta{=}z], \label{eq3-lem1}
\]
where $\mu_3[\eta\mid\zeta{=}z]$
is the third conditional central moment,
\[
\mu_3[\eta\mathrel| \zeta{=}z] = \ME \bigl[ \bigl(\eta- \ME[\eta
\mathrel| \zeta{=}z] \bigr)^3 \bigm| \zeta{=}z \bigr] .
\]
\end{lemma}
\endgroup
\begin{proof}\begingroup\abovedisplayskip=12pt\belowdisplayskip=12pt
We have
\[
p_\zeta(z) = \ME p_\xi(z-\eta) = \frac{1}{\sqrt{2\pi}} \ME
\econst^{-\frac{1}2 (z-\eta)^2} .
\]
Then
\begin{align}
p'_\zeta(z) &= \frac{1}{\sqrt{2\pi}} \ME \bigl[ (\eta- z)
\econst^{-\frac{1}2 (z-\eta)^2} \bigr],
\nonumber
\\[4pt]
\frac{\der}{\der z} \bigl(\ln p_\zeta(z) \bigr) &= \frac{p'_\zeta(z)}{p_\zeta(z)} =
\frac{\ME[(\eta- z) \econst^{-\frac{1}2 (z-\eta)^2}]}{
\ME\econst^{-\frac{1}2 (z-\eta)^2}} = \frac{\ME\eta\econst^{-\frac{1}2 (z-\eta)^2}}{
\ME\econst^{-\frac{1}2 (z-\eta)^2}} - z,
\nonumber
\\[4pt]
\frac{\der^2}{\der z^2} \bigl(\ln p_\zeta(z) \bigr) &= \frac
{\ME\eta^2 \econst^{-\frac{1}2 (z-\eta)^2}
\ME\econst^{-\frac{1}2 (z-\eta)^2} -
 ( \ME\eta\econst^{-\frac{1}2 (z-\eta)^2}  )^2}{
 ( \ME\econst^{-\frac{1}2 (z-\eta)^2}  )^2} -
1,
\nonumber
\\
\frac{\der^3}{\der z^3} \bigl(\ln p_\zeta(z) \bigr) &= \bigl(\ME
\econst^{-\frac{1}2 (z-\eta)^2} \bigr)^{-3}
\nonumber
\\
&\quad\times \bigl( \ME \bigl[\eta^2(\eta-z) \econst^{-\frac{1}2 (z-\eta)^2}
\bigr] \bigl(\ME\econst^{-\frac{1}2 (z-\eta)^2} \bigr)^2
\nonumber
\\
& \quad+ \ME\eta^2 \econst^{-\frac{1}2 (z-\eta)^2} \ME \bigl[(\eta- z)
\econst^{-\frac{1}2 (z-\eta)^2} \bigr] \ME \econst^{-\frac{1}2 (z-\eta)^2}
\nonumber
\\
&\quad- 2 \ME \bigl[\eta (\eta-z) \econst^{-\frac{1}2 (z-\eta)^2} \bigr] \ME\eta
\econst^{-\frac{1}2 (z-\eta)^2} \ME\econst^{-\frac{1}2 (z-\eta)^2}
\nonumber
\\
&\quad- 2 \ME\eta^2 \econst^{-\frac{1}2 (z-\eta)^2} \ME \econst^{-\frac{1}2 (z-\eta)^2}
\ME \bigl[(\eta-z) \econst^{-\frac{1}2 (z-\eta)^2} \bigr]
\nonumber
\\
& \quad+ 2 \bigl(\ME\eta\econst^{-\frac{1}2 (z-\eta)^2} \bigr)^2 \ME \bigl[(
\eta-z) \econst^{-\frac{1}2 (z-\eta)^2} \bigr] \bigr)
\nonumber
\\
& = \bigl(\ME\econst^{-\frac{1}2 (z-\eta)^2} \bigr)^{-3} \times \bigl( \ME
\eta^3 \econst^{-\frac{1}2 (z-\eta)^2} \bigl(\ME\econst^{-\frac{1}2 (z-\eta)^2}
\bigr)^2
\nonumber
\\
&\quad - 3 \ME\eta^2 \econst^{-\frac{1}2 (z-\eta)^2} \ME\eta
\econst^{-\frac{1}2 (z-\eta)^2} \ME\econst^{-\frac{1}2 (z-\eta)^2} + 2 \bigl(\ME\eta
\econst^{-\frac{1}2 (z-\eta)^2} \bigr)^3 \bigr) . \label{eq-de3p}
\end{align}\endgroup

If $\eta$ has a pdf, the conditional pdf of $\eta$ given $\zeta{=}z$ is
equal to
\[
p_{\eta|\zeta=z}(y) = \frac{p_\eta(y) \econst^{-\frac{1}2 (z-y)^2}}{
\ME\econst^{-\frac{1}2 (z-\eta)^2}} ;
\]
otherwise, we can use the conditional density of $\eta$ w.r.t.\@
marginal density
\[
\frac{\der\cdf_{\eta|\zeta=z}(y)}{
\der\cdf_{\eta}(y)} = \frac{\econst^{-\frac{1}2 (z-y)^2}}{
\ME\econst^{-\frac{1}2 (z-\eta)^2}} .
\]
Anyway, the conditional moments of $\eta$ given $\zeta{=}z$
are equal
to
%
\begin{equation}
\ME \bigl[\eta^k \bigm|\zeta{=}z \bigr] = \frac{\ME\eta^k \econst^{-\frac{1}2 (z-\eta)^2}}{
\ME\econst^{-\frac{1}2 (z-\eta)^2}}.
\label{eq-condmoments}
\end{equation}

From \eqref{eq-de3p} and \eqref{eq-condmoments} it follows that
\begin{align*}
\frac{\der^3}{\der z^3} \bigl(\ln p_\zeta(z) \bigr) &= \ME \bigl[
\eta^3 \bigm| \zeta{=}z \bigr] - 3 \ME \bigl[\eta^2
\bigm| \zeta{=}z \bigr] \ME[\eta \mathrel| \zeta{=}z] + 2 \bigl( \ME[\eta
\mathrel| \zeta{=}z] \bigr)^3
\\
&= \mu_3[\eta \mathrel| \zeta{=}z] . \qedhere
\end{align*}
\end{proof}

\begin{corollary}
\label{lemg-d2lnpm}
Let $\xi$ and $\eta$ be independent random variables such that
$\xi\sim N(\mu, \sigma^2)$.
Denote $\zeta= \xi+ \eta$,
and denote the pdf of $\zeta$ by $p_\zeta(z)$.
Then
\[
\frac{\der^3}{\der z^3} \bigl(\ln p_\zeta(z) \bigr) = \frac{1}{\sigma^6} \,
\mu_3[\eta\mid\zeta{=}z] .
\]
\end{corollary}

\begin{lemma}
\label{lem-m3p}
Assume that the distribution of a random variable $X$ satisfies the following
conditions:
\begin{enumerate}
\renewcommand{\labelenumi}{{\rm\theenumi)}}
\item$X$ has a continuously differentiable density $p_X(x)$.
\item$X$ is unimodal in the following sense:
there exists a mode $\varmode\in\mathbb{R}$ such that for all $x\in
\mathbb{R}$,
we have the equality $\sign(p'_X(x)) = \sign(\varmode-x)$.
\item Whenever $x_1 < \varmode< x_2$ and $p_X(x_1) = p_X(x_2)$,
then $p_X(x_1) > - p_X(x_2)$.
\item$\ME|X|^3 < \infty$.
\end{enumerate}
Then $\mu_3(X) := \ME(X - \ME X)^3 > 0$.\vadjust{\goodbreak}
\end{lemma}

\begin{proof}\begingroup\abovedisplayskip=6pt\belowdisplayskip=6pt
\xch{1)}{1.} $\ME X > \varmode$.
Denote by $x_1(z)$ and $x_2(z)$ the solutions to the equation\break  $p_X(x)=z$
(see \xch{Fig.~\ref{fig:1}}{Fig.~\ref{fig:4xy}}):
\begin{align*}
x_1(z) < \varmode< x_2(z) &\quad{\mbox{if}}\ 0 < z<
\max (p_X);
\\
x_1(z) = \varmode= x_2(z) &\quad{\mbox{if}}\ z =
\max(p_X);
\\
p_X \bigl(x_1(z) \bigr) = p_X
\bigl(x_2(z) \bigr) = z &\quad{\mbox{if}} \ 0 < z \le
\max(p_X).
\end{align*}
Represent the expectation as a double integral and change the order of
integration:
\begin{align}
\ME X &= \varmode+ \int_{-\infty}^{\infty} (x-\varmode)
p_X(x)\, \der x
\nonumber
\\
&= \varmode+ \iint\limits
_{\{(x,z) \,|\, 0 \le z \le p_X(x)\}} (x-\varmode)\,\der x\,\der z
\nonumber
\\
&= \varmode+ \int_0^{\max(p_X)} \Biggl( \int
_{x_1(z)}^{x_2(z)} (x-\varmode) \, \der x \Biggr) \der z
\nonumber
\\
&= \varmode+ \int_0^{\max(p_X)} \frac{(x_2(z) - \varmode)^2 - (\varmode- x_1(z))^2}{2}\,
\der z . \label{eq-L617}
\end{align}

%
\begin{figure*}
\includegraphics{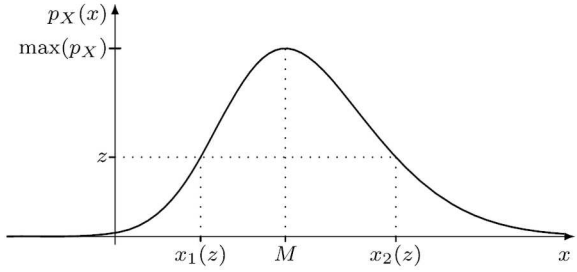}
\caption{To proof of Lemma~\ref{lem-m3p}, \xch{part~1)}{part~1}. Sample $p_X(x)$ and
definition of $x_1(z)$ and $x_2(z)$}\label{fig:1}
\end{figure*}

For all $x_2>\varmode$, by the implicit function theorem,
\[
\frac{\der}{\der x_2} x_1 \bigl(p_X(x_2)
\bigr) = \frac{p'_X(x_2)}{p'_X(x_1(p_X(x_2)))} > -1
\]
because $p_X(x_1(p_X(x_2))) = p_X(x_2)$ implies
$p'_X(x_1(p_X(x_2))) > -p'_X(x_2) > 0$.
Note that
$x_1(p_X(\varmode)) = \varmode$.
By the Lagrange theorem,
\[
x_1 \bigl(p_X(x_2) \bigr) = \varmode+
(x_2 - \varmode) \cdot \frac{\der}{\der x_3} x_1
\bigl(p_X(x_3) \bigr) \Big|_{x_3 = \varmode+ (x_2 - \varmode)\theta}
\]
for some $\theta\in(0,1)$;
\begin{align*}
x_1 \bigl(p_X(x_2) \bigr) &> \varmode-
(x_2 - \varmode) \quad\mbox{for}\ x_2 > \varmode;
\\
x_1(z) &> \varmode- \bigl(x_2(z) - \varmode \bigr)
\quad \mbox{for}\ 0<z<\max(p_X);
\\
x_2(z) - \varmode&> \varmode- x_1(z) > 0;
\\
\frac{(x_2(z) - \varmode)^2}{2} &> \frac{(\varmode- x_1(z))^2}{2};
\end{align*}
the last integrand in (\ref{eq-L617}) is positive,
and then (\ref{eq-L617}) implies $\ME X > \varmode$.

\xch{2)}{\paragraph{2}}
Consider the function
\[
f(t) = p_X(\ME X + t) - p_X(\ME X - t) ,
\]
which is odd and strictly decreasing on the interval $[-(\ME X -
\varmode),\; \ME X - \varmode]$.
Therefore, $f(t)$ attains 0 only once on this interval, that is, at the
point \xch{0 (see Fig. \ref{fig:2}).}{0.}

\begin{figure*}
\includegraphics{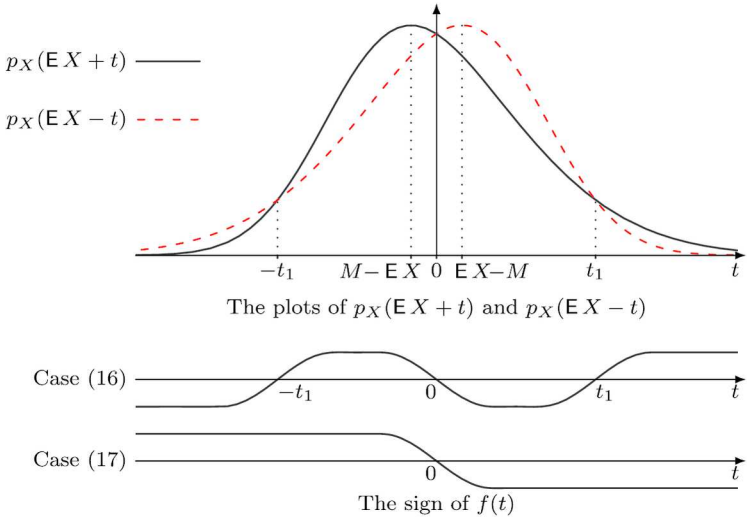}
\caption{To proof of Lemma~\ref{lem-m3p}, \xch{part~2)}{part~2}}\label{fig:2}
\end{figure*}

If $t> \ME X - \varmode$ (more generally, $|t| > \ME X - \varmode$)
and $f(t)=0$, then $f'(t) = p'_X(\ME X + t) + p'_X(\ME X - t) > 0$
by condition~{3)} of Lemma~\ref{lem-m3p}.
Therefore, $f(t)$ can attain $0$ only once on $(\ME X - \varmode,\;
+\infty)$,
and if it attains $0$ (say, at a~point $t_1 > \ME X - \varmode> 0$),
it is increasing in the neighborhood of $t_1$.

Hence, there may be two cases of sign changing of \xch{$f(t)$ (Fig. \ref{fig:2}).}{$f(t)$.} Either
%
\begin{equation}
\label{case-mp2p-2a} \exists t_1>0\ \forall x{\in}\mathbb{R} \; : \;
\sign \bigl(f(t) \bigr) = \sign(t) \sign\bigl(|t| - t_1\bigr),
\end{equation}
or
%
\begin{equation}
\label{case-mp2p-2b} \forall x{\in}\mathbb{R}\; : \; \sign \bigl(f(t) \bigr) = -
\sign(t) .
\end{equation}

\xch{3)}{\paragraph{3}}
We have
\begin{align}
0 &= \ME[X - \ME X] = \int_{-\infty}^\infty(x - \ME X)
p_X(x)\, \der x
\nonumber
\\
& = \int_{-\infty}^\infty t \, p_X(\ME X + t)
\, \der t
\nonumber
\\
& = \int_0^\infty t \, p_X(\ME X + t)
\, \der t + \int_0^\infty(-t) \, p_X(
\ME X - t)\, \der t
\nonumber
\\
& = \int_0^\infty t \,f(t) \, \der t, \label{eq-m3p-p3}
\end{align}
where $f(t)$ is defined in the second part of the proof.

Note that the case (\ref{case-mp2p-2b}) is impossible
because otherwise the last integrand in (\ref{eq-m3p-p3}) would be negative
and thus the integral could not be equal to 0.

\xch{4)}{\paragraph{4}}
Similarly to (\ref{eq-m3p-p3}),
\[
\ME(X - \ME X)^3 = \int_0^\infty
t^3 f(t) \, \der t.
\]
Subtract $t_1^2$ times Eq.~(\ref{eq-m3p-p3}), where
$t_1$ comes from (\ref{case-mp2p-2a}):
\[
\ME(X - \ME X)^3 = \int_0^\infty t
\bigl(t^2 - t_1^2 \bigr) f(t) \, \der t.
\]
The integrand is positive for $t>0$, $t\neq t_1$,
and hence $\mu_3[X] = \ME(X - \ME X)^3 > 0$.
\endgroup\end{proof}

\begin{lemma}
\label{lem-key-ineq4}
For all $x\in\mathbb{R}$ and $\sigma^2 \ge0$,
\begin{align*}
&\sign \bigl(L_4 \bigl(x,\sigma^2 \bigr)
L_1 \bigl(x,\sigma^2 \bigr)^2 - 3
L_3 \bigl(x, \sigma^2 \bigr) L_2 \bigl(x,
\sigma^2 \bigr) L_1 \bigl(x,\sigma^2 \bigr) +
2 L_2 \bigl(x, \sigma^2 \bigr)^3 \bigr)
\\
&\quad = \sign(x) .
\end{align*}
\end{lemma}

Lemma~\ref{lem-aux-L4L3L2L1} is needed to prove Lemma~\ref
{lem-key-ineq4}. The notation $F(y)$ and $y_0$ is common for
Lemmas \ref{lem-key-ineq4} and \ref{lem-aux-L4L3L2L1}.

For fixed $x>0$ and $\sigma^2$, consider the function
%
\begin{equation}
\label{eq-def-F} F(y) = \ln \biggl( \frac{\econst^y}{(\econst^y + 1)^2} \biggr) -
\frac{(y-x)^2}{2\sigma^2} .
\end{equation}
Its derivative
\[
F'(y) = 1 - 2 \frac{\econst^y}{\econst^y + 1} - \frac{y-x}{\sigma^2}
\]
is strictly decreasing, and
\[
\lim_{y\to-\infty} F'(y) = +\infty, \qquad \lim
_{y\to+\infty} F'(y) = -\infty.
\]
Hence, $F'(y)$ attains $0$ at a unique point. Denote this point by
$y_0$, and then
%
\begin{equation}
\label{eq-def-y0} \sign \bigl(F'(y) \bigr) = - \sign(y -
y_0).
\end{equation}

\begin{lemma}
\label{lem-aux-L4L3L2L1}
For the function $F(y)$ defined in \eqref{eq-def-F},
for $y_0$ satisfying~\eqref{eq-def-y0},
and for $y_3$ and $y_4$ such that
$F'(y_3) + F'(y_4) = 0$ and $y_3 < y_4$,
we have the following inequalities:
\begin{enumerate}
\renewcommand{\labelenumi}{{\rm\theenumi)}}
\item$y_3 < y_0 < y_4$ and $F'(y_3) = -F'(y_4) > 0$.
\item$y_3 + y_4 > 0$\xch{.}{;}
\item\label{lem-aux-part3}$F''(y_3) < F''(y_4) < 0$\xch{.}{;}
\item$F(y_3) > F(y_4)$.\vadjust{\goodbreak}
\end{enumerate}
\end{lemma}

\begin{proof}
\xch{1)}{1.} \textit{The inequality} \xch{$y_3 < y_0 < y_4$}{\mathversion{bold}$y_3 < y_0 < y_4$}
is a consequence of \eqref{eq-def-y0},
and \eqref{eq-def-y0} implies {\mathversion{normal}$F'(y_3) > 0$}.

\xch{2) $y_3 + y_4 > 0$.}{\paragraph{2. \mathversion{bold}$y_3 + y_4 > 0$}}
For all $y\in\mathbb{R,}$
%
\[
F'(y) + F'(-y) = \frac{2 x}{\sigma^2} > 0 .
\]
Since $F'(y_3) + F'(-y_3) > 0$ and $F'(y_3) + F'(y_4) = 0$,
we have $F'(-y_3) > F'(y_4)$, and then $-y_3 < y_4$ because the derivative
$F'(y)$ is decreasing.

\xch{3) $F''(y_3) < F''(y_4) < 0$.}{\paragraph{3. \mathversion{bold}$F''(y_3) < F''(y_4) < 0$}}
The second derivative
\[
F''(y) = \frac{-2 \econst^y}{(\econst^y + 1)^2} - \frac{1}{\sigma^2}
\]
is an even function strictly increasing on $[0, +\infty)$
and attaining only negative values.

The inequalities $y_3 < y_4$ and $y_3 + y_4 > 0$ can be rewritten as
$|y_3| < y_4$, and then
\[
F''(y_3) = F''\bigl(|y_3|\bigr)
< F''(y_4) < 0 .
\]

\xch{4) $F(x_3) > F (x_4)$.}{\paragraph{4. \mathversion{bold}$F(x_3) > F (x_4)$}}
Consider the inverse function
\[
\bigl(F' \bigr)^{-1}(t), \quad t {\in}\mathbb{R}.
\]
Its derivative is
\[
\frac{\der}{\der t} \bigl( \bigl(F' \bigr)^{-1} (t)
\bigr) = \frac{1}{F''((F')^{-1} (t))} < 0 .
\]
Then
\begin{align*}
\frac{\der}{\der t} \bigl( F \bigl( \bigl(F' \bigr)^{-1}
(t) \bigr) \bigr) &= \frac{F'((F')^{-1} (t))} {F''((F')^{-1} (t))} = \frac{t} {F''((F')^{-1} (t))} ;
\\
\frac{\der}{\der t} \bigl( F \bigl( \bigl(F' \bigr)^{-1}
(t) \bigr) - F \bigl( \bigl(F' \bigr)^{-1} (-t) \bigr)
\bigr) &= \frac{t}{F''((F')^{-1} (t))} + \frac{-t}{F''((F')^{-1} (-t))} .
\end{align*}

Apply already proven \xch{part \ref{lem-aux-part3})}{part \ref{lem-aux-part3}} of Lemma~\ref{lem-aux-L4L3L2L1}.
If $t>0$, then $(F')^{-1}(t) < (F')^{-1}(-t)$ (because $(F')^{-1}(t)$
is a decreasing function) and
$F'((F')^{-1} (t)) + F'((F')^{-1} (-t)) = t - t = 0$.
Then by \xch{part \ref{lem-aux-part3})}{part \ref{lem-aux-part3}}
\[
F'' \bigl( \bigl(F'
\bigr)^{-1} (t) \bigr) < F'' \bigl(
\bigl(F' \bigr)^{-1} (-t) \bigr) < 0, \quad t>0.
\]
Hence,
\[
\frac{\der}{\der t} \bigl( F \bigl( \bigl(F' \bigr)^{-1}
(t) \bigr) - F \bigl( \bigl(F' \bigr)^{-1} (-t) \bigr)
\bigr) > 0, \quad t>0 .
\]
Note that
\[
F \bigl( \bigl(F' \bigr)^{-1} (0) \bigr) - F \bigl(
\bigl(F' \bigr)^{-1} (-0) \bigr) = 0.
\]
By the Lagrange theorem, for $t>0$,
%
\begin{equation}
F \bigl( \bigl(F' \bigr)^{-1} (t) \bigr) - F \bigl(
\bigl(F' \bigr)^{-1} (-t) \bigr) = t \cdot
\frac{\der}{\der t_1} \bigl( F \bigl( \bigl(F' \bigr)^{-1}
(t_1) \bigr) - F \bigl( \bigl(F' \bigr)^{-1}
(-t_1) \bigr) \bigr) > 0, \label{eq-L868}
\end{equation}
where the derivative is taken at some point $t_1 \in(0,t)$.

Substituting $t=F'(y_3)>0$ (then $-t = F'(y_4)$), we obtain
\mbox{$F(y_3) - F(y_4) > 0$}.
\end{proof}

\begin{proof}[Proof of Lemma~\ref{lem-key-ineq4}]
\textit{Case 1. $x>0$ and $\sigma^2>0$.}
Recall that for fixed $\sigma^2$, $L_1(x,\sigma^2)$ is the pdf
of $\eta+ \xi$, where $\eta$ and $\xi$ are independent variables,
$\Prob(\eta<y) = \frac{\xch{\econst}{e}^y}{\xch{\econst}{e}^y+1}$ and $\xi\sim N(0,\sigma^2)$
(see Appendix~\ref{Appendix-A}).
By Corollary~\ref{lemg-d2lnpm},
%
\begin{equation}
\label{eq-L724} \frac{\der^3}{\der x^3} \bigl(\ln L_1 \bigl(x,
\sigma^2 \bigr) \bigr) = \frac{1}{\sigma^6} \mu_3[\eta
\mathrel| \eta{+}\xi{=}x],
\end{equation}
but
%
\begin{equation}
\label{eq-L730} \frac{\der^3}{\der x^3} \bigl(\ln L_1 \bigl(x,
\sigma^2 \bigr) \bigr) = \frac
{L_4 L_1^2 - 3 L_3 L_2 L_1 + 2 L_2^3} {L_1^3},
\end{equation}
where $L_k$ are evaluated at the point $(x,\sigma^2)$.
Since $L_1(x,\sigma^2) > 0$, we have to prove that
$\mu_3[\eta\mathrel| \eta{+}\xi{=}x] > 0$.
Therefore, we apply Lemma~\ref{lem-m3p}.

The pdf of the conditional distribution of $\eta$ given
$\eta+ \xi= x$ is equal to
\[
p_{\eta| \eta+ \xi= x}(y) = \frac{1}{
 \ME \econst^{- \frac{(\eta- x)^2}{2\sigma^2} } } \cdot \frac{\econst^y}{(1+\econst^y)^2}
\econst^{- \frac{(y-x)^2}{2\sigma^2}
} .
\]
The pdf $p_{\eta| \eta+ \xi= x} (y)$ is continuously differentiable.
The conditional distribution has a finite $k$th moment because
$y^k \econst^{-\frac{(y-x)^2}{2\sigma^2}}$ is bounded
for any $k \in\mathbb{N}$. Hence, 
conditions 1) and 4) of Lemma~\ref{lem-m3p} are satisfied.

Evaluate
%
\[
\ln p_{\eta| \eta+ \xi= x} (y) = \ln \biggl( \frac{\econst^y}{(\econst^y + 1)^2} \biggr) -
\frac{y - x}{2\sigma^2} - \ln \bigl( \ME\econst^{-\frac{(\eta-x)^2}{2\sigma^2}} \bigr) = F(y) + C,
\]
where the function $F(y)$ is defined in \eqref{eq-def-F}, and
$C = -\ln ( \ME\exp ( - \frac{(\eta- x)^2}{2\sigma^2} )
 )$
depends only on $x$ and $\sigma^2$ and does not depend on $y$.

We check condition 2) of Lemma~\ref{lem-m3p}:
\begin{align}
\nonumber
p_{\eta| \eta+ \xi= x}(y) &= \xch{\econst}{e}^{F(y) + C};
\\
\label{eq-dpdF} \frac{\der}{\der y} p_{\eta| \eta+ \xi= x}(y) &= F'(y)
\xch{\econst}{e}^{F(y) + C};
\\
\nonumber
\sign \biggl( \frac{\der}{\der y} p_{\eta| \eta+ \xi= x}(y) \biggr) &= \sign
\bigl(F'(y) \bigr) = - \sign(y-y_0),
\end{align}
and condition {2)} holds with $\varmode= y_0$,
where $y_0$ is defined just above \eqref{eq-def-y0}.

Now check condition of 3) of Lemma~\ref{lem-m3p}.
The proof is illustrated by Fig.~\ref{fig:3}.
Assume that $p_{\eta| \eta+ \xi= x} (y_1) = p_{\eta| \eta+ \xi=
x} (y_2)$
and $y_1 < y_0 < y_2$.
Then $F(y_1) = F(y_2)$.

Denote
\[
y_4 = \bigl(F' \bigr)^{-1}
\bigl(-F'(y_1) \bigr).
\]
Then
$F'(y_1) + F'(y_4) = F'(y_1) - F'(y_1) = 0$,
and by \eqref{eq-def-y0}, as $y_1 < y_0$, we have
$F'(y_1)>0$, $F'(y_4)<0$,
$y_4 > y_0 > y_1$.
By Lemma~\ref{lem-aux-L4L3L2L1}, $F(y_1) > F(y_4)$.
%
\begin{figure*}
\includegraphics{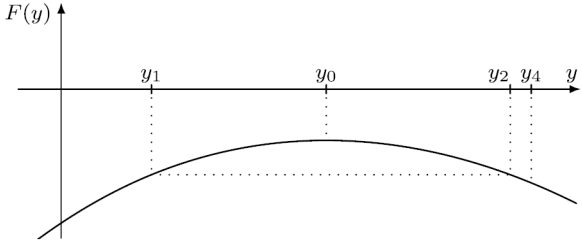}
\caption{To proof of Lemma~\ref{lem-key-ineq4}. Checking condition 3)~of
Lemma~\ref{lem-m3p}}\label{fig:3}
\end{figure*}

Hence, $F(y_2) = F(y_1) > F(y_4)$.
Because the function $F(y)$ is decreasing on
$( y_0, +\infty)$ (see \eqref{eq-def-y0}),
we have $y_2 < y_4$.
Since the function $F'(y)$ is decreasing,
$F'(y_2) > F'(y_4) = - F'(y_1)$, which implies $F'(y_1) + F'(y_2) > 0$.
By (\ref{eq-dpdF}) we have
$p'_{\eta| \eta+ \xi= x} (y_1) + p'_{\eta| \eta+ \xi= x} (y_2) >
0$.\vspace*{2pt}

All the conditions of Lemma~\ref{lem-m3p}
are satisfied. By Lemma~\ref{lem-m3p}, ${\mu_3[\eta\mathrel| \eta+\xi=x]}>0$,
and by (\ref{eq-L724})--(\ref{eq-L730}),
%
\begin{equation}
L_4 \bigl(x,\sigma^2 \bigr) L_1 \bigl(x,
\sigma^2 \bigr)^2 - 3 L_3 \bigl(x,
\sigma^2 \bigr) L_2 \bigl(x,\sigma^2 \bigr)
L_1 \bigl(x,\sigma^2 \bigr) + 2 L_2 \bigl(x,
\sigma^2 \bigr) > 0 \label{neq-L908}
\end{equation}
for all $x>0$ and $\sigma^2>0$.

\xch{\textit{Case 2. $x\le0$ and $\sigma^2 > 0$.}}{\paragraph{\mathversion{normal}Case 2. $x\le0$ and $\sigma^2 > 0$}}
The distribution of $\eta+ \xi$ is symmetric.
Hence, $L_1(x,\sigma^2)$ and $L_3(x,\sigma^2)$ are
even functions in $x$, and
$L_2(x,\sigma^2)$ and $L_4(x,\sigma^2)$ are
odd functions in $x$. Then
\[
L_4 \bigl(x,\sigma^2 \bigr) L_1 \bigl(x,
\sigma^2 \bigr)^2 - 3 L_3 \bigl(x,
\sigma^2 \bigr) L_2 \bigl(x,\sigma^2 \bigr)
L_1 \bigl(x,\sigma^2 \bigr) + 2 L_2 \bigl(x,
\sigma^2 \bigr)^3
\]
is an odd function in $x$. It is equal to 0 for $x=0$,
and it is negative for $x<0$ by Case~1; see (\ref{neq-L908}).

\xch{\textit{Case 3. $\sigma^2 = 0$.}}{\paragraph{\mathversion{normal}Case 3. $\sigma^2 = 0$}}
The function $L_1(x,0)$ is the pdf of the logistic distribution,
and $L_{k+1}(x,0)$ is its $k$th derivative:
\begin{align*}
L_1(x,0) &= \frac{\econst^x}{(1+\econst^x)^2}; \qquad L_2(x,0) =
\frac{\econst^x(1-\econst^x)}{(1+\econst^x)^3};
\\
L_3(x,0) &= \frac{\econst^x}{(1+\econst^x)^4} \bigl(1 - 4 \econst^x +
\econst ^{2x} \bigr);
\\
L_4(x,0) &= \frac{\econst^x(1-\econst^x)}{(1+\econst^x)^5} \bigl(1 - 10 \econst^x +
\econst^{2x} \bigr) .
\end{align*}
Then
\begin{align*}
L_4 L_1^2 - 3 L_3
L_2 L_1 + 2 L_2^3 &=
\frac{\xch{\econst}{e}^{3x} (1-\xch{\econst}{e}^x)} {(1+\xch{\econst}{e}^x)^9} \bigl(-2 \xch{\econst}{e}^x \bigr);
\\
\sign \bigl(L_4 L_1^2 - 3 L_3
L_2 L_1 + 2 L_2^3 \bigr) &=
\sign(x),
\end{align*}
where $L_k$ are evaluated at the point $(x,0)$.

Lemma~\ref{lem-key-ineq4} is proven.
\end{proof}
%

%

\end{document}